\documentclass{amsart}

\usepackage{bbm}
\usepackage{amssymb}
\usepackage{hyperref}

\newcommand*{\mailto}[1]{\href{mailto:#1}{\nolinkurl{#1}}}

\newtheorem{theorem}{Theorem}[section]

\newtheorem{lemma}[theorem]{Lemma}
\newtheorem{proposition}[theorem]{Proposition}
\newtheorem{corollary}[theorem]{Corollary}

\newcommand{\R}{{\mathbb R}}
\newcommand{\N}{{\mathbb N}}

\newcommand{\C}{{\mathbb C}}

\newcommand{\spr}[2]{\langle #1 , #2 \rangle}

\newcommand{\I}{\mathrm{i}}

\newcommand{\supp}{\mathrm{supp}}
\newcommand{\linspan}{\mathrm{span}}
\newcommand{\indik}{\mathbbm{1}}

\newcommand{\trace}[1]{\mathrm{tr}(#1)}
\renewcommand{\det}[1]{\mathrm{det}(#1)}

\newcommand{\M}{M}
\newcommand{\edge}{{e}}
\newcommand{\fedge}{{d}}
\newcommand{\bedge}{{b}}
\newcommand{\edges}{\mathcal{E}}
\newcommand{\G}{\mathcal{G}}
\newcommand{\CoupM}{R}
\newcommand{\repc}{\mathsf{Y}}

\newcommand{\Trf}{\mathsf{T}}

\renewcommand{\labelenumi}{(\roman{enumi})}
\renewcommand{\theenumi}{\roman{enumi}}
\numberwithin{equation}{section}

\begin{document}

\title{An inverse spectral problem for a star graph of Krein strings}

\author[J.\ Eckhardt]{Jonathan Eckhardt}
\address{Institut Mittag-Leffler\\
Aurav\"agen 17\\ SE-182 60 Djursholm\\ Sweden}
\email{\mailto{jonathaneckhardt@aon.at}}

\thanks{\href{http://dx.doi.org/10.1515/crelle-2014-0003}{J.\ Reine Angew.\ Math.\ {\bf 715} (2016), 189--206}}
\thanks{{\it Research supported by the AXA Research Fund under the Mittag-Leffler Fellowship Project}}


\begin{abstract}
 We solve an inverse spectral problem for a star graph of Krein strings, where the known spectral data comprises the spectrum associated with the whole graph, the spectra associated with the individual edges as well as so-called coupling matrices.  
 In particular, we show that these spectral quantities uniquely determine the weight within the class of Borel measures on the graph, which give rise to trace class resolvents. 
 Furthermore, we obtain a concise characterization of all possible spectral data for this class of weights. 
\end{abstract}

\maketitle

\section{Introduction}

 {\sf O}ne of the origins of spectral theory lies in the investigation of vibrating strings: 
  The equations describing small transversal vibrations of an inhomogeneous string, which is clamped between its endpoints at $a$ and $b$ are given by 
  \begin{align}\label{eqnWE}
     \omega(x)  u_{tt}(x,t) & = u_{xx}(x,t), & u(a,t) &  = u(b,t) = 0, 
  \end{align}
  where $u$ represents the displacement of the string. 
 Hereby $\omega$ denotes the mass density of the string, which is stretched by a unit force. 
 Employing the common separation of variables method for this equation, one is led to the spectral problem
\begin{align}\label{eqnSP}
 - f''(x) = z\, \omega(x) f(x), \quad x\in(a,b),
\end{align}
 which is of fundamental importance for solving the wave equation~\eqref{eqnWE}. 

 {\sf M}ark Krein was the first one to deal with a corresponding inverse spectral problem (which is why~\eqref{eqnSP} is usually termed Krein string). 
   Given the spectrum associated with~\eqref{eqnSP} together with so-called {\em norming constants}, he managed to find the corresponding weight density $\omega$. 
  In particular, he was the first to notice that an elementary case of this inverse problem can be solved by virtue of the theory on continued fractions by Thomas Stieltjes.
  Thus, this particular case, where $\omega$ is a finite sum of weighted Dirac measures, is usually referred to as Stieltjes strings. 
  The solution of the corresponding inverse spectral problem can be written down explicitly in terms of the finite set of spectral data. 
  Utilizing this, the inverse problem in the general case can then be solved by approximating the spectral data and employing continuity and compactness arguments. 
  For several surveys on direct and inverse spectral theory for Krein strings we refer to \cite{dymc76}, \cite{gakr41}, \cite{kakr74}, \cite{ka94}, \cite{kowa82}.  

 {\sf A}nother kind of inverse spectral problem is the so-called {\em inverse three-spectra problem}. 
  Hereby, the norming constants are replaced by two additional spectra, which are obtained by clamping the string at some interior point $c\in(a,b)$.  
  This type of inverse problem was first introduced in \cite{pi99} for Schr\"odinger operators (see also \cite{dr09}, \cite{dr10}, \cite{gesi99}, \cite{pi06}). 
  Following this, corresponding inverse problems have been investigated in \cite{DistInverse}, \cite{hrmy03} for Sturm--Liouville operators with distributional potentials, in \cite{gesi97}, \cite{mite04} for Jacobi operators and in \cite{alhrmy07} for operators corresponding to certain oscillating systems.
  In the context of Krein strings, the inverse three-spectra problem has been solved in \cite{bopi08} for the class of Stieltjes strings and in \cite{ThreeSpectra} for the class of Krein strings with trace class resolvents. 
  Let us remark that there is typically an issue with non-uniqueness, already noticed in \cite{gesi99}. 
  In order to overcome this problem, one has to introduce additional spectral quantities; see \cite{ThreeSpectra}, \cite{mite04}. 

 {\sf R}egarding the interval $(a,b)$ as a two-edged graph with central vertex $c$, the inverse three-spectra problem can also be viewed as an inverse spectral problem for a star graph. 
 From this perspective, the given spectral data comprise the spectrum associated with the whole graph as well as the spectra associated with the two edges of the graph. 
 Of course, this immediately suggests a natural generalization of the inverse three-spectra problem to general star graphs with finitely many edges. 
 This problem has been dealt with in \cite{pi00}, \cite{pi07} for Schr\"odinger operators and in \cite{bopi08a}, \cite{pirotr13} for Stieltjes strings.
 In the present article, we will give a complete solution of the corresponding inverse spectral problem for the class of Krein strings with trace class resolvents. 
 At this point, let us mention that there are of course also several other inverse problems on graphs. 
 One of the most delicate questions in this respect is, to which extent the structure of the graph can be read off from suitable spectral data; \cite{gusm01}, \cite{kuno05}. 
 In particular, the so-called {\em boundary control method} has been employed to investigate these kind of problems in \cite{avku08}, \cite{avkuno10}, \cite{be04}, \cite{bewa09}. 
 The recovery of the differential operators on the edges of an a priori known graph, using various spectral data, has also been studied in \cite{brwe05}, \cite{cuwa07}, \cite{ge88}, \cite{yu05}.

 {\sf S}etting the stage, let $\G\subseteq\R^2$ be a planar star graph with finitely many edges. 
  We think of this graph as a system of inhomogeneous strings, joined at the central vertex of the graph and clamped at the outer vertices. 
  The (vertical) displacement of such a web of strings can be described by a real-valued function $u$ on $\G$ and its mass density by a non-negative function (respectively, a Borel measure) $\omega$ on $\G$.  
  The equations describing the motion of this system are analogous to the first one in~\eqref{eqnWE} for every edge of the graph, together with a Kirchhoff type interface condition at the central vertex and Dirichlet boundary conditions at the outer vertices. 
  To be precise, this is not quite true since the stretching forces might (and for some graphs, depending on its angles, will have to) be different for the individual edges.
  However, since there are no additional difficulties arising from this issue, we will assume that all our strings are stretched by a unit force for simplicity of notation. 

 {\sf V}ia a usual separation of variables ansatz, one is led to the spectral problem 
 \begin{align}\label{eqnDESG}
  - f''(x) = z\, \omega(x) f(x), \quad x\in\G, 
 \end{align}
 on the whole graph $\G$. 
 Of course, it is not obvious at all how this differential equation has to be interpreted. 
 Roughly speaking, it has its usual meaning on each of the individual edges augmented by a Kirchhoff type interface condition at the central vertex (the precise notion will be given in Section~\ref{s2}). 
 
 {\sf O}ur aim in this article is to solve the following inverse spectral problem: 
   Given the spectrum associated with~\eqref{eqnDESG} as well as all the spectra associated with the individual edges of the graph, we try to find the corresponding weight density  $\omega$ on the graph $\G$. 
  In the special case of Stieltjes strings, this problem has been solved quite recently in \cite{bopi08a}, \cite[Section~2]{pirotr13}, where the solution can again be written down explicitly in terms of the spectral data.
  Unsurprisingly, this solution turned out to be unique only if the spectra are assumed to be disjoint. 

 {\sf L}acking uniqueness for this inverse spectral problem, we are urged to introduce additional spectral quantities, which guarantee a unique solution. 
  For the inverse three-spectra problem, this can be done by so-called {\em coupling constants}, relating the norms of eigenfunctions on the two subintervals. 
  In the case of a general star graph, we need somewhat more elaborate spectral quantities, which relate the norms of eigenfunctions on the individual edges. 
  This will lead us to the introduction of the so-called {\em coupling matrices}, containing precisely this information.
  Augmented by them, our inverse spectral problem will turn out to be uniquely solvable. 
   
 {\sf T}he content of the present article can be outlined as follows: 
  In Section~\ref{s2} we will set up a precise notion of the differential equation~\eqref{eqnDESG} on a star graph $\G$, where $\omega$ is a non-negative Borel measure on $\G$ satisfying a particular growth condition near the outer vertices. 
  Consequently, we will introduce the associated self-adjoint linear relation in the Sobolev space $H_0^1(\G)$ as well as those relations associated with the individual edges of the graph. 
  This is somewhat in contrast to most of the existing literature (for example, see \cite{dymc76}, \cite{kakr74} for intervals and \cite{be04}, \cite{kaklvowe09} for graphs) which usually considers the spectral problem in a weighted $L^2(\G;\omega)$ Hilbert space. 
  Consequently, after deriving few necessary conditions for our spectral data, we will introduce the coupling matrices in Section~\ref{s3}. 
  The following section is then devoted to the statement and the proof of the main result of this article; the solution of the aforementioned inverse spectral problem for our class of weight measures. 
  Concluding, we will show in Section~\ref{s6} that it is always possible to approximate the solution of this inverse problem by Stieltjes strings, which are obtained by cutting off the spectral data in a suitable way.
 
  {\sf A}s it will be used rather frequently, let us recall a special case of the integration by parts formula for Borel measures. For positive $x>0$, it  reads  
    \begin{align}\label{eqnIntParts}
     \int_{(0,x)}  h(t)^\ast g(t) d\upsilon(t) =  h(x)^\ast \int_{(0,x)} g(t) d\upsilon(t) - \int_{0}^{x} h'(t)^\ast \int_{(0,t)} g(s)d\upsilon(s)\, dt,
    \end{align}
   where $\upsilon$ is a Borel measure on $\R$, $h$ is a locally absolutely continuous function on $\R$ and $g$ is a locally integrable (with respect to $\upsilon$) function on $\R$.

\section{A star graph of Krein strings}\label{s2}

To set the stage, let $\G\subseteq\R^2$ be a planar star graph consisting of a central vertex $c$ and finitely many edges of finite length attached to it. 
Note that our graph $\G$ is supposed not to contain the outer vertices, rendering $\G$ non-compact with respect to the topology inherited from $\R^2$. 
The set of all edges will be denoted with $\edges$ and every edge $\edge\in\edges$ is assumed to have finite length $l_\edge>0$. 
For simplicity, each edge $\edge\in\edges$ will be identified with an open interval $I_\edge=(0,l_\edge)$, where zero corresponds to the central vertex $c$. 
Consequently, it is also possible to identify $\G$ with the (disjoint) union of the intervals $I_\edge$, $\edge\in\edges$ augmented by the central vertex $c$. 
In view of this identification, the following notation will be quite convenient: 
Given a function $f$ on the graph $\G$, its restriction to the interval $I_\edge$ (identified with the edge $\edge\in\edges$) is denoted with $f_\edge$. 
Moreover, we will denote with $f_\edge(0)$ the limit of $f_\edge(x)$ as $x\rightarrow 0$ in $I_\edge$ (as long as it exists). 
Finally, we also introduce the abbreviations 
\begin{align}
 \frac{1}{L} & = \sum_{\edge\in\edges} \frac{1}{l_\edge}, & \frac{1}{L_\edge} & = \sum_{\fedge\not=\edge} \frac{1}{l_\fedge} = \frac{1}{L} - \frac{1}{l_\edge}, \quad \edge\in\edges. 
\end{align}

We denote with $C_0(\G)$ the space of all complex-valued, continuous functions on the graph $\G$, which tend to zero near all of the outer vertices.   
Furthermore, the Sobolev space $H_0^1(\G)$ on the graph $\G$ consists of all functions $f\in C_0(\G)$ such that $f_\edge\in H^1(I_\edge)$ for every edge $\edge\in\edges$. 
Equipped with the (definite) inner product 
\begin{align}
 \spr{f}{g}_{H_0^1(\G)} = \sum_{\edge\in\edges} \int_0^{l_\edge} f_\edge'(x) g_\edge'(x)^\ast dx, \quad f,\, g\in H_0^1(\G),
\end{align}
the Sobolev space $H_0^1(\G)$ turns into a reproducing kernel Hilbert space. 
In particular, the function $\repc\in H_0^1(\G)$ such that 
\begin{align}
 \spr{f}{\repc}_{H_0^1(\G)} = f(c), \quad f\in H_0^1(\G),
\end{align}
is given by
\begin{align}
 \repc_\edge(x)  = L \biggl( 1-\frac{x}{l_\edge}\biggr), \quad x\in I_\edge.
\end{align} 
Also observe that one has the decomposition
\begin{align}\label{eqnDecomp}
 H_0^1(\G) = \bigoplus_{\edge\in\edges} H_0^1(I_\edge) \, \oplus \linspan\lbrace \repc\rbrace,
\end{align}
upon regarding the Sobolev spaces $H_0^1(I_\edge)$, $\edge\in\edges$ as subspaces of $H_0^1(\G)$.

As the final preparatory step, we introduce the class $\M(\G)$ of all (non-negative) Borel measures $\omega$ on the graph $\G$, which satisfy the growth condition  
\begin{align}\label{eqnGrowth} 
 \int_\G \repc(x) d\omega(x)  = L \sum_{\edge\in\edges} \int_{I_\edge} \biggl(1-\frac{x}{l_\edge} \biggr) d\omega_\edge(x) < \infty
\end{align}
near the outer vertices of the graph. 
Similarly as above, the quantity $\omega_\edge$ denotes hereby the restriction of $\omega$ to $I_\edge$ (identified with the edge $\edge\in\edges$). 
In this respect, also note that all measures in $\M(\G)$ are regular (in the sense of \cite[Definition~2.15]{ru74}). 

We now turn to the main object of this article; the differential equation 
\begin{align}\label{eqnDE}
 - f'' = g\, \omega
\end{align}
on the graph $\G$, where $\omega$ is some arbitrary fixed measure in $\M(\G)$.  
Of course, it is not obvious at all how this equation has to be understood.
To be precise, for some given function $g$, which is locally integrable with respect to $\omega$, we say the function $f$ is a solution of~\eqref{eqnDE} on the edge $\edge\in\edges$ if $f_\edge$ is locally absolutely continuous with 
\begin{align}\label{eqnDEedge}
 f_\edge'(x) = F_\edge - \int_{(0,x)} g_\edge(t) d\omega_\edge(t),
\end{align}
for some $F_\edge\in \C$ and almost all $x\in I_\edge$. 
In this case, we will denote for definiteness with $f_\edge'$ the unique left-continuous representative of the derivative of $f_\edge$ on $I_\edge$ such that~\eqref{eqnDEedge} holds for all $x\in I_\edge$ and $f_\edge'(0) = F_\edge$. 
Next we say that $f$ is a solution of~\eqref{eqnDE} on the graph $\G$ if it is a solution of~\eqref{eqnDE} on every edge, it is continuous in the central vertex $c$ and furthermore  satisfies the interface condition
\begin{align}\label{eqnIntCond}
 g(c) \omega(\lbrace c\rbrace) + \sum_{\edge\in\edges} f_\edge'(0) & = 0 
\end{align}
there. 
Hereby, the limits in~\eqref{eqnIntCond} are known to exist because of~\eqref{eqnDEedge}. 

The differential equation~\eqref{eqnDE} on the graph $\G$ gives rise to a linear relation $S$ in the Sobolev space $H_0^1(\G)$ defined by 
\begin{align}
 S = \lbrace (f,g)\in H_0^1(\G)\times H_0^1(\G) \,|\, -f'' = g\, \omega \text{ on the graph }\G \rbrace.
\end{align}
Alternatively, we could have also employed a weak formulation of the differential equation, as the following useful characterization shows.

\begin{proposition}\label{propWeakForm}
 Some pair $(f,g)\in H_0^1(\G)\times H_0^1(\G)$ belongs to $S$ if and only if 
 \begin{align}\label{eqnWeakForm}
   \spr{f}{h}_{H_0^1(\G)}= \int_\G g(x) h(x)^\ast d\omega(x)
 \end{align}
 for every $h\in H_0^1(\G)$. 
\end{proposition}

\begin{proof}
If the pair $(f,g)$ belongs to $S$, then employing the differential equation, respectively~\eqref{eqnDEedge}, on some edge $\edge\in\edges$ and integrating by parts yields  
\begin{align*}
 \int_0^{x} f_\edge'(t) h_\edge'(t)^\ast dt & = f_\edge'(0) (h_\edge(x)^\ast - h(c)^\ast) + \int_{(0,x)} g_\edge(t) h_\edge(t)^\ast d\omega_\edge(t) \\ 
                                                             & \qquad\qquad\qquad\qquad\qquad\qquad\qquad - \int_{(0,x)} g_\edge(t) h_\edge(x)^\ast d\omega_\edge(t)                                                    
\end{align*}
for $x\in I_\edge$ and every $h\in H_0^1(\G)$. 
To see that the last term converges to zero as $x\rightarrow l_\edge$, we may apply Lebesgue's dominated convergence theorem. 
In fact, the integrand converges pointwise to zero as $x\rightarrow l_\edge$ and can be bounded by   
\begin{align*}
 |h_\edge(x) g_\edge(t) \indik_{(0,x)}(t)| & \leq \sqrt{l_\edge-x}\, \|h\|_{H_0^1(\G)} \sqrt{l_\edge-t}\, \|g\|_{H_0^1(\G)}  \indik_{(0,x)}(t) \\
                                                                & \leq (l_\edge-t) \|h\|_{H_0^1(\G)} \|g\|_{H_0^1(\G)}
\end{align*}
 for all $x$, $t\in I_\edge$. 
The limits of the remaining terms are obvious (also note that the function $|gh|$ is always integrable with respect to $\omega$). 
Upon utilizing the interface condition~\eqref{eqnIntCond}, a simple calculation shows that~\eqref{eqnWeakForm} holds for every $h\in H_0^1(\G)$. 

In order to prove the converse, first fix some edge $\edge\in\edges$, let $x_0\in I_\edge$ and consider the function $h\in H_0^1(I_\edge)$ given by 
\begin{align*}
 h_\edge(x)^\ast = \begin{cases} \int_{0}^x \left( f_\edge'(t) + \int_{(0,t)} g_\edge(s)d\omega_\edge(s) + C \right) dt, & x\in (0,x_0], \\ 0, & x\in(x_0,l_\edge),
                              \end{cases}
\end{align*}
where the constant $C$ is necessarily given by   
\begin{align*}
 C = - \frac{1}{x_0} \int_{0}^{x_0} \biggl( f_\edge'(x) + \int_{(0,x)} g_\edge(t) d\omega_\edge(t) \biggr) dx.
\end{align*}
Now we integrate  the right-hand side of~\eqref{eqnWeakForm} by parts to obtain 
\begin{align*}
 0 & = \int_0^{l_\edge} h_\edge'(x)^\ast \biggl( f'_\edge(x) + \int_{(0,x)} g_\edge(t)d\omega_\edge(t) + C \biggr) dx \\
    & =  \int_{0}^{x_0} \biggl| f_\edge'(x) + \int_{(0,x)} g_\edge(t)d\omega_\edge(t) + C \biggr|^2 dx.
\end{align*}
Thus the integrand vanishes almost everywhere on $(0,x_0)$, which shows that $f$ is a solution of $-f'' = g\, \omega$ on the edge $\edge$ since $x_0\in I_\edge$ was arbitrary. 
Finally, after another integration by parts (also using the fact that the differential equation holds on the individual edges), we find that 
 \begin{align*}
   f(c) & = \spr{f}{\repc}_{H_0^1(\G)} = \int_\G g(x) \repc(x) d\omega(x) 
          = f(c) + L \biggl(g(c) \omega(\lbrace c\rbrace) + \sum_{\edge\in\edges} f_\edge'(0)\biggr), 
 \end{align*}
 which establishes the interface condition~\eqref{eqnIntCond} and thus the claim.
\end{proof}

From this alternative characterization one sees that the linear relation $S$ is possibly multi-valued if $\omega$ is not supported on the whole graph.
In fact, some $g\in H_0^1(\G)$ belongs to the multi-valued part of $S$ if and only if it vanishes almost everywhere with respect to $\omega$. 
Nevertheless, the linear relation $S$ turns out to be self-adjoint.

\begin{theorem}\label{thmSA}
The linear relation $S$ is self-adjoint. 
\end{theorem}

\begin{proof}
There is a bounded self-adjoint operator $R$ on $H_0^1(\G)$ such that 
\begin{align*}
  \spr{Rg}{h}_{H_0^1(\G)} = \int_\G g(x) h(x)^\ast d\omega(x), \quad g,\, h\in H_0^1(\G),
\end{align*}
since the right-hand side is a bounded symmetric sesquilinear form on $H_0^1(\G)$. 
In view of Proposition~\ref{propWeakForm}, we may identify $R$ with the inverse of $S$.  
\end{proof}

It remains to introduce the linear relations $S_\edge$ in $H_0^1(I_\edge)$ associated with the differential equation~\eqref{eqnDE} on some edge $\edge\in\edges$, which are given by  
\begin{align}
 S_\edge = \lbrace (f,g)\in H_0^1(I_\edge) \times H_0^1(I_\edge) \,|\, -f'' = g\, \omega \text{ on the edge }\edge \rbrace.  
\end{align} 
Similarly as above, some pair $(f,g)\in H_0^1(I_\edge) \times H_0^1(I_\edge)$ belongs to $S_\edge$ if and only if 
 \begin{align}\label{eqnWeakFormEdge}
   \spr{f}{h}_{H_0^1(I_\edge)}= \int_{I_\edge} g_\edge(x) h_\edge(x)^\ast d\omega_\edge(x)
 \end{align}
 for every $h\in H_0^1(I_\edge)$. 
  In fact, this can be viewed as a special case of Proposition~\ref{propWeakForm}, upon regarding the interval $I_\edge$ as a two-edged star graph. 
 Consequently, the linear relations $S_\edge$ turn out to be self-adjoint as well. 
 Due to the growth restriction~\eqref{eqnGrowth} on $\omega$, the inverses of $S_\edge$ are trace class operators \cite{ka94}, \cite[Proposition~2.3]{ThreeSpectra} with 
\begin{align}\label{eqnTraceEdge}
 \trace{S_\edge^{-1}} = \int_{I_\edge} x \biggl(1-\frac{x}{l_\edge}\biggr) d\omega_\edge(x).
\end{align}
Moreover, these inverses are obviously non-negative since the sesquilinear form on the right-hand side of~\eqref{eqnWeakFormEdge} is positive semidefinite. 

Similar facts also hold for the linear relation $S$ associated with the whole graph. 
In order to be able to write down the trace of the inverse of $S$ in a simple way, we introduce the function $\Trf\in H_0^1(\G)$ such that 
\begin{align}
 \Trf_\edge(x) = L \biggl( 1+ \frac{x}{L_\edge} \biggr) \biggl( 1-\frac{x}{l_\edge} \biggr), \quad x\in I_\edge. 
\end{align}
Let us mention that this function indeed is the diagonal of the Green's function at zero energy, that is, the integral kernel of the operator $S^{-1}$. 

\begin{proposition}\label{propTrace}
 The inverse of  $S$ is a non-negative trace class operator with  
 \begin{align}\label{eqnTraceForm}
  \trace{S^{-1}} = \int_\G \Trf(x) d\omega(x).
 \end{align}
\end{proposition}

\begin{proof}
First one observes that for every edge $\edge\in\edges$ and $g\in H_0^1(I_\edge)$ we have  
\begin{align*}
 \spr{S^{-1} g}{g}_{H_0^1(\G)} = \int_{I_\edge} g_{\edge}(x) g_{\edge}(x)^\ast d\omega_\edge(x) = \spr{S_\edge^{-1} g}{g}_{H_0^1(I_\edge)}.
\end{align*}
Since the operators $S_\edge^{-1}$ are trace class, we infer that $S^{-1}$ is trace class as well with
\begin{align*}
 \trace{S^{-1}}  = \frac{\spr{S^{-1} \repc}{\repc}_{H_0^1(\G)}}{\|\repc\|_{H_0^1(\G)}^2} + \sum_{\edge\in\edges} \trace{S_\edge^{-1}},
\end{align*}
in view of the decomposition~\eqref{eqnDecomp}. 
Upon employing~\eqref{eqnTraceEdge} as well as evaluating   
\begin{align*}
 \frac{\spr{S^{-1} \repc}{\repc}_{H_0^1(\G)}}{\|\repc\|_{H_0^1(\G)}^2} = L\, \omega(\lbrace c \rbrace) + L \sum_{\edge\in\edges} \int_{I_\edge} \biggl( 1-\frac{x}{l_\edge}\biggr)^2 d\omega_\edge(x),
\end{align*}
it is a simple calculation to derive the trace formula~\eqref{eqnTraceForm}. 
Finally, the inverse of $S$ is non-negative because the sesquilinear form on the right-hand side of~\eqref{eqnWeakForm} is positive semidefinite.  
\end{proof}

\section{The coupling matrices}\label{s3}

Before we state the solution of the inverse problem, we first need to derive a few necessary conditions and to introduce additional spectral quantities. 
To this end, we fix again some weight measure $\omega\in\M(\G)$ and denote all associated quantities as in the preceding section.
The condition~\eqref{eqnGrowth} on the growth of $\omega$ guarantees the existence of solutions to the differential equation~\eqref{eqnDE} on every edge $\edge\in\edges$, which vanish at the outer vertex; see for example \cite[Section~2]{ThreeSpectra}, \cite[Theorem~9]{ka67}.

\begin{theorem}\label{thmPhiE}
 For every edge $\edge\in\edges$ and $z\in\C$ there is a unique solution $\phi_\edge(z,\cdot\,)$ of the differential equation $-f'' = z f \, \omega$ on the edge $\edge$ such that
 \begin{align}
  \phi_\edge(z,x) & \sim \biggl(1 - \frac{x}{l_\edge}\biggr), &  \phi_\edge'(z,x) & \sim - \frac{1}{l_\edge},
 \end{align} 
 as $x\rightarrow l_\edge$ in $I_\edge$. Moreover, for each $x\in [0,l_\edge)$, the functions $\phi_\edge(\,\cdot\,,x)$ and $\phi_\edge'(\,\cdot\,,x)$ are real entire and of exponential type zero.
\end{theorem}

Some complex $\mu\in\C$ is an eigenvalue of the linear relation $S_\edge$ if and only if it is a zero of the entire function $\phi_\edge(\,\cdot\,,0)$. 
In fact, this function is indeed the characteristic function of the linear relation $S_\edge$ (see \cite[Theorem~2.4]{ThreeSpectra}), that is,  
\begin{align}
 \phi_\edge(z,0) = \det{I-zS_\edge^{-1}} = \prod_{\mu\in\sigma(S_\edge)} \biggl( 1-\frac{z}{\mu}\biggr), \quad z\in\C.
\end{align}
For the linear relation $S$ on the whole graph, a similar role is played by the real entire function $W$, which is given by 
\begin{align}\label{eqnW}
  - W(z) = L \biggl(  \omega(\lbrace c\rbrace) z  + \sum_{\edge\in\edges} \frac{\phi_\edge'(z,0)}{\phi_\edge(z,0)}  \biggr) \prod_{\edge\in\edges} \phi_\edge(z,0), \quad z\in\C.
\end{align}
Observe that in the case of a graph with two edges, this function reduces (up to a scalar multiple) to the usual Wronskian of the solutions given in Theorem~\ref{thmPhiE}. 
Before we show that $W$ is the characteristic function of the linear relation $S$ indeed, we mention the following useful identity 
\begin{align}\label{eqnDotPhi}
 - \dot{\phi}_\edge(\mu,0) \phi_\edge'(\mu,0) = \int_{I_\edge} |\phi_\edge(\mu,x)|^2 d\omega_\edge(x)\not=0, \quad \mu\in\sigma(S_\edge), 
\end{align}
where the dot denotes differentiation with respect to the spectral parameter. 
This relation can be deduced immediately from \cite[Lemma~2.2]{ThreeSpectra}.

\begin{proposition}\label{propMult}
 The entire function $W$ is the characteristic function of the linear relation $S$, that is,  
 \begin{align}\label{eqnWProd}
  W(z) = \det{I-zS^{-1}} = \prod_{\lambda\in\sigma(S)} \biggl( 1-\frac{z}{\lambda} \biggr)^{\kappa_\lambda}, \quad z\in\C,
 \end{align} 
 where the multiplicity $\kappa_\lambda$ of every eigenvalue $\lambda\in\sigma(S)$ is given by 
 \begin{align}\label{eqnkappalambda}
  \kappa_\lambda = \begin{cases} 1, & \text{if }\phi_\edge(\lambda,0)\not=0 \text{ for all edges }\edge\in\edges, \\ k - 1, & \text{if }\phi_\edge(\lambda,0) = 0 \text{ for precisely }k\text{ edges }\edge\in\edges. \end{cases}
 \end{align} 
\end{proposition}

\begin{proof} 
 First, suppose that $\lambda\in\C$ is such that $\edges_\lambda$ is empty, where we introduced the set $\edges_\lambda = \lbrace \edge\in\edges \,|\, \phi_\edge(\lambda,0) = 0\rbrace$ for notational simplicity. 
 If $\lambda$ is an eigenvalue of the linear relation $S$, then every associated eigenfunction $\psi$ is necessarily of the form 
\begin{align}\label{eqnPsiE}
 \psi_\edge(x) = \alpha_\edge \phi_\edge(\lambda,x), \quad x\in I_\edge,~\edge\in\edges, 
\end{align}
for some $\alpha_\edge\in\C$, $\edge\in\edges$. 
Since $\edges_\lambda$ is empty, we infer that $\psi(c)\not=0$ as well as
\begin{align}\label{eqnPsiC}
 \psi_\edge(x) = \frac{\psi(c)}{\phi_\edge(\lambda,0)} \phi_\edge(\lambda,x), \quad x\in I_\edge,~\edge\in\edges, 
\end{align}
because $\psi$ is continuous at the central vertex. 
 Thus we see that the eigenvalue $\lambda$ is simple. 
 Furthermore, since $\psi$ has to satisfy an interface condition similar to the one in~\eqref{eqnIntCond} at the central vertex, we have  
 \begin{align}\label{eqnWIntCond}
  \lambda\, \omega(\lbrace c\rbrace) + \sum_{\edge\in\edges} \frac{\phi_\edge'(\lambda,0)}{\phi_\edge(\lambda,0)} = 0,
 \end{align} 
 and hence $W(\lambda)=0$. 
 Conversely, if $\lambda$ is a zero of the entire function $W$, then~\eqref{eqnW} gives~\eqref{eqnWIntCond}  and thus every function $\psi$ given by~\eqref{eqnPsiC} with some non-zero $\psi(c)$
 is an eigenfunction of the linear relation $S$ with eigenvalue $\lambda$. 
 In order to show that $\lambda$ is a simple zero of $W$, one first observes that for every edge $\edge\in\edges$ one gets    
 \begin{align}\label{eqnPhiHN}
   \frac{\phi_\edge'(z,0)}{\phi_\edge(z,0)} - \frac{\phi_\edge'(z^\ast,0)}{\phi_\edge(z^\ast,0)}  = (z-z^\ast) \int_{I_\edge} \left| \frac{\phi_\edge(z,x)}{\phi_\edge(z,0)} \right|^2 d\omega_\edge(x), \quad z\in\C\backslash\sigma(S_\edge), 
 \end{align}
 upon employing the Lagrange identity.  
 Consequently, we find that 
 \begin{align*}
  \left.\frac{\partial}{\partial z} \frac{\phi_\edge'(z,0)}{\phi_\edge(z,0)} \right|_{z=\lambda} = \lim_{z\rightarrow\lambda} \int_{I_\edge} \left| \frac{\phi_\edge(z,x)}{\phi_\edge(z,0)} \right|^2 d\omega_\edge(x) = \int_{I_\edge} \left| \frac{\phi_\edge(\lambda,x)}{\phi_\edge(\lambda,0)} \right|^2 d\omega_\edge(x),
 \end{align*}
 where  necessary bounds on the integrands can be found in the proof of \cite[Theorem 2.1]{ThreeSpectra}. 
    Now the derivative of the function in brackets in~\eqref{eqnW} at $\lambda$ is given by 
    \begin{align*}
     \omega(\lbrace c\rbrace) + \sum_{\edge\in\edges} \left.\frac{\partial}{\partial z} \frac{\phi_\edge'(z,0)}{\phi_\edge(z,0)} \right|_{z=\lambda} = \frac{1}{\lambda} \sum_{\edge\in\edges} \int_0^{l_\edge} \left| \frac{\phi_\edge'(\lambda,x)}{\phi_\edge(\lambda,0)}\right|^2 dx 
    \end{align*}
   and thus strictly positive.
   Hereby, we employed the identities   
  \begin{align*}
    \lambda \int_{I_\edge} \left| \frac{\phi_\edge(\lambda,x)}{\phi_\edge(\lambda,0)} \right|^2 d\omega_\edge(x) =  \frac{\phi_\edge'(\lambda,0)}{\phi_\edge(\lambda,0)} + \int_0^{l_\edge} \left| \frac{\phi_\edge'(\lambda,x)}{\phi_\edge(\lambda,0)} \right|^2 dx, \quad \edge\in\edges,
 \end{align*}    
 which are obtained by an integration by parts, as well as the interface condition~\eqref{eqnWIntCond}. 
 But this ensures that $\lambda$ is a simple zero of $W$.  

Next, suppose that the set $\edges_\lambda$ is not empty. 
If $\lambda$ is an eigenvalue of $S$ and $\psi$ is an associated eigenfunction, then the coefficients in~\eqref{eqnPsiE} necessarily satisfy
\begin{align}\label{eqnalphacond}
 \sum_{\edge\in\edges_\lambda} \alpha_\edge \phi_\edge'(\lambda,0) & = 0, & \alpha_\edge = 0, \quad\edge\not\in\edges_\lambda. 
\end{align}
 In this case, the set $\edges_\lambda$ contains at least two edges since otherwise we would obtain a contradiction in the first equation of~\eqref{eqnalphacond}. 
Therefore, the product on the right-hand side in~\eqref{eqnW} has at least a double zero in $\lambda$ and hence $W$ vanishes in $\lambda$. 
Moreover, since all coefficients $\alpha_\edge\in\C$, $\edge\in\edges$ which satisfy~\eqref{eqnalphacond} give rise to an eigenfunction $\psi$ given by~\eqref{eqnPsiE}, we see that the multiplicity $\kappa_\lambda$ of the eigenvalue $\lambda$ is given as in the claim. 
Conversely, if $\lambda$ is a zero of $W$, then note that the residue 
 \begin{align}\label{eqnStar}
  \lim_{z\rightarrow\lambda} (z-\lambda) \sum_{\edge\in\edges} \frac{\phi_\edge'(z,0)}{\phi_\edge(z,0)} = \sum_{\edge\in\edges_\lambda} \frac{\phi_\edge'(\lambda,0)}{\dot{\phi}_\edge(\lambda,0)} = - \sum_{\edge\in\edges_\lambda} \frac{1}{\lambda} \frac{\|\phi_\edge(\lambda,\cdot\,)\|_{H_0^1(I_\edge)}^2}{\dot{\phi}_\edge(\lambda,0)^{2}} 
 \end{align}
 is strictly negative, where we employed equation~\eqref{eqnDotPhi} and~\eqref{eqnWeakFormEdge}. 
 In view of~\eqref{eqnW}, this shows that the multiplicity of the zero $\lambda$ is one less than its multiplicity as a zero of the product on the right-hand side of~\eqref{eqnW}, that is, equal to $\kappa_\lambda$. 
 As a consequence, at least two edges belong to $\edges_\lambda$ and hence an eigenfunction $\psi$ of $S$ can be simply provided by choosing an arbitrary solution of~\eqref{eqnalphacond}. 

 In conclusion, we have shown that the zeros of $W$ are precisely the eigenvalues of $S$ with the multiplicities being the same and given as in the claim.
 This immediately gives the second equality in~\eqref{eqnWProd} and also the first one upon noting that $W$ is of exponential type zero by Theorem~\ref{thmPhiE} with $W(0)=1$.  
\end{proof}

As already mentioned, we will need one more necessary condition on the spectra of our star graph of Krein strings, which will be given in the following proposition. 
In order to state it, recall that a Herglotz--Nevanlinna function is an analytic function which maps the open upper complex half-plane into itself or is a real constant. 
For further information regarding this kind of functions, we refer for example to~\cite[Chapter~VI]{akgl93}, \cite[Section~2]{gets00}, \cite[Section~7.2]{le96}, \cite[Chapter~5]{roro94}.

\begin{proposition}\label{propGHN}
 The meromorphic function $G$, given by 
 \begin{align}\label{eqnGDC}
    G(z) = \frac{L}{W(z)} \prod_{\edge\in\edges} \phi_\edge(z,0), \quad z\in\C\backslash\sigma(S),
 \end{align}
 is a Herglotz--Nevanlinna function.
\end{proposition}

\begin{proof}
 Just observe that one has the expansion 
 \begin{align}\label{eqnGinvExp}
      - G(z)^{-1}= \omega(\lbrace c\rbrace) z + \sum_{\edge\in\edges} \frac{\phi_\edge'(z,0)}{\phi_\edge(z,0)}, \quad z\in\C\backslash\R,
 \end{align}
 where each summand is a Herglotz--Nevanlinna function; cf.\ \eqref{eqnPhiHN}.
\end{proof}

It is known (see for example \cite[Theorem~2.1]{gesi99}, \cite[Theorem~27.2.1]{le96}) that the zeros and poles of any meromorphic Herglotz--Nevanlinna function are necessarily simple and interlacing. 
Thus Proposition~\ref{propGHN} shows that the merged spectra $\sigma(S_\edge)$, $\edge\in\edges$ interlace the spectrum $\sigma(S)$ after suitable cancelation. 
 Moreover, there must not be a zero of the function~\eqref{eqnGDC} which is less or equal to the smallest eigenvalue of $S$, that is, one has for every edge $\edge\in\edges$ 
 \begin{align}\label{eqnSmallEV}
 \inf\sigma(S) < \inf\sigma(S_\edge).
 \end{align} 
 Finally, let us only briefly mention that in combination with Proposition~\ref{propMult}, it is also possible to infer restrictions on the multiplicity of two consecutive eigenvalues of $S$; cf.\ \cite[Corollary~2.6]{pirotr13}. 

 We already mentioned that the spectra are in general not enough to obtain a unique solution to our inverse spectral problem. 
 For this reason, we have to introduce additional spectral quantities; cf.\ \cite[Corollary~5.2]{ThreeSpectra}.
 To this end, we set  
 \begin{align}\label{eqnedgeslambda}
  \edges_\lambda = \lbrace \edge\in\edges \,|\, \phi_\edge(\lambda,0) = 0 \rbrace
 \end{align}
 (as in the proof of Proposition~\ref{propMult}) and introduce the matrix
\begin{align}\label{eqnCoupMatrixdef}
 \Gamma_\lambda = \left( \Gamma_{\lambda,\edge\fedge} \right)_{\edge, \fedge\in\edges_\lambda} = \left( \frac{\|\phi_\edge(\lambda,\cdot\,)\|_{H_0^1(I_\edge)}^2}{\|\phi_\fedge(\lambda,\cdot\,)\|_{H_0^1(I_\fedge)}^2} \frac{\dot{\phi}_\fedge(\lambda,0)^2}{\dot{\phi}_\edge(\lambda,0)^2}  \right)_{\edge, \fedge\in\edges_\lambda}
\end{align}
for every eigenvalue $\lambda\in\sigma(S)$ for which the set $\edges_\lambda$ is not empty. 
Obviously, the matrix  $\Gamma_\lambda$ belongs to the class $\CoupM(\edges_\lambda)$ of real matrices, which is defined by     
\begin{align}\label{eqnCoupMdef}
 \CoupM(\edges_\lambda) = \lbrace \left( r_{\edge\fedge}\right)_{\edge, \fedge\in\edges_\lambda} |\, 0 < r_{\edge\fedge} = r_{\edge\bedge} r_{\bedge\fedge}, \text{ for all } \edge, \fedge, \bedge\in\edges_\lambda \rbrace. 
\end{align}
Every matrix $r$ from this class clearly satisfies $r_{\edge\edge}=1$ as well as $r_{\edge\fedge} = r_{\fedge\edge}^{-1}$ for all edges $\edge$, $\fedge\in\edges_\lambda$. 
For future purposes, we furthermore state the useful identity
\begin{align}\label{eqnRid}
 \sum_{\edge\in\edges_\lambda} \left( \,\sum_{\fedge\in\edges_\lambda} r_{\fedge\edge} \right)^{-1} = \sum_{\edge\in\edges_\lambda} \left( r_{\bedge\edge} \sum_{\fedge\in\edges_\lambda} r_{\fedge\bedge} \right)^{-1} = \sum_{\edge\in\edges_\lambda} r_{\edge\bedge} \left( \,\sum_{\fedge\in\edges_\lambda} r_{\fedge\bedge} \right)^{-1} = 1. 
\end{align}
Also note that a matrix from the class $\CoupM(\edges_\lambda)$ is determined by $\kappa_\lambda -1$ of its entries. 
 For example, one may take the entries $r_{\edge\fedge}$ for $\fedge\in\edges_\lambda\backslash\lbrace \edge\rbrace$, where $\edge\in\edges_\lambda$ is some fixed edge.
In particular, if the set $\edges_\lambda$ only consists of two edges, then the matrix $\Gamma_\lambda$ is determined by a single positive real number, which resembles the coupling constant used for the inverse three-spectra problem \cite[Corollary~5.2]{ThreeSpectra}. 
Because of this, the matrix $\Gamma_\lambda$ is referred to as the {\em coupling matrix} corresponding to the eigenvalue $\lambda$. 

The need for additional spectral quantities stems from the fact that we are not able to recover the individual summands in~\eqref{eqnGinvExp} as soon as some of them have common poles. 
In this case, their residues will merge but can not be recovered from the (known) function on the left-hand side of~\eqref{eqnGinvExp}. 
However, the coupling matrices provide us with information about the ratios of the contribution of every summand to the residue. 
More precisely, if $\lambda\in\sigma(S)$ is an eigenvalue such that the set $\edges_\lambda$ is not empty, then a calculation using~\eqref{eqnStar} shows that one has 
\begin{align}\label{eqnGresCM}
  \lim_{z\rightarrow\lambda} - \frac{z-\lambda}{G(z)} = \lim_{z\rightarrow \lambda} (z-\lambda) \frac{\phi_\edge'(\lambda,0)}{\phi_\edge(\lambda,0)} \, \sum_{\fedge\in\edges_\lambda} \Gamma_{\lambda,\fedge\edge}
\end{align}
for every edge $\edge\in\edges_\lambda$. 
This means that under additional knowledge of the coupling matrices, we are able to recover the residues of the individual summands on the right-hand side of~\eqref{eqnGinvExp} from the (known) function $G$.

\section{An inverse spectral problem}

We are now ready to state the main result of this article; the solution of an inverse spectral problem for a star graph of Krein strings. 
As our principal spectral data, we take the spectrum $\sigma(S)$ associated with the whole graph as well as all the spectra $\sigma(S_\edge)$ associated with the individual edges $\edge\in\edges$. 
Roughly speaking, the main necessary and sufficient conditions for solvability of the corresponding inverse problem within the class $\M(\G)$ is the trace class property in Proposition~\ref{propTrace} together with the Herglotz--Nevanlinna property of the meromorphic function in Proposition~\ref{propGHN}. 
However, if the spectra are not pairwise disjoint, then we also need to prescribe the corresponding coupling matrices $\Gamma_\lambda$ as introduced in~\eqref{eqnCoupMatrixdef}, in order to guarantee uniqueness. 
It turns out that indeed all matrices from the class $\CoupM(\edges_\lambda)$, defined in~\eqref{eqnCoupMdef}, appear as coupling matrices. 

\begin{theorem}\label{thmIP}
 Let $\sigma$ and $\sigma_\edge$ for every edge $\edge\in\edges$ be discrete sets of positive reals and let $\Pi_\lambda\in\CoupM(\edges_\lambda)$ for every $\lambda\in\sigma$  with $\edges_\lambda\not=\emptyset$, where  $\edges_\lambda = \lbrace \edge\in\edges \,|\, \lambda\in\sigma_\edge\rbrace$. 
 Set   
 \begin{align}
  \kappa_\lambda = \begin{cases} 1, & \lambda \text{ belongs to none of the sets }\sigma_\edge,~\edge\in\edges, \\ k-1, & \lambda \text{ belongs to }\sigma_\edge\text{ for precisely }k\text{ edges }\edge\in\edges,\end{cases}
 \end{align}
 and suppose that $\kappa_\lambda\geq 1$ for every $\lambda\in\sigma$, that the sums
 \begin{align}
  \sum_{\lambda\in\sigma} & \frac{1}{\lambda}, &  \sum_{\mu\in\sigma_\edge} & \frac{1}{\mu}, \quad \edge\in\edges, 
 \end{align}
 are finite and that the meromorphic function 
 \begin{align}\label{eqnProdHN}
  \prod_{\lambda\in\sigma} \biggl( 1-\frac{z}{\lambda}\biggr)^{-\kappa_\lambda} \prod_{\edge\in\edges} \prod_{\mu\in\sigma_\edge} \biggl( 1-\frac{z}{\mu}\biggr), \quad z\in\C\backslash\sigma, 
 \end{align}
 is a Herglotz--Nevanlinna function. Then there is a unique measure $\omega\in\M(\G)$ on the graph $\G$ such that the spectrum associated with $\omega$ on the whole graph is $\sigma$, the spectra associated with $\omega$ on the individual edges $\edge\in\edges$ are $\sigma_\edge$ and the coupling matrices associated with $\omega$ are $\Pi_\lambda$ for every eigenvalue $\lambda\in\sigma$ with $\edges_\lambda\not=\emptyset$. 
\end{theorem}

\begin{proof}
{\em Existence.} 
For notational simplicity, we introduce the entire functions
\begin{align*}
 V(z) & = \prod_{\lambda\in\sigma} \biggl( 1-\frac{z}{\lambda}\biggr)^{\kappa_\lambda}, & P_\edge(z) & = \prod_{\mu\in\sigma_\edge} \biggl( 1-\frac{z}{\mu}\biggr), \quad z\in\C, 
\end{align*}
for every edge $\edge\in\edges$, as well as the (strictly negative by assumption) residues 
\begin{align*}
  - \frac{1}{\eta_\mu} & = \lim_{z\rightarrow\mu} -(z-\mu) \frac{V(z)}{L} \prod_{\edge\in\edges} P_\edge(z)^{-1} <0, \quad \mu\in\bigcup_{\edge\in\edges} \sigma_\edge.\end{align*}
Now let $\omega\in\M(\G)$ and denote all quantities associated with this measure as in the preceding sections. 
We are going to show that one may redefine $\omega_\edge$ for every edge $\edge\in\edges$ as well as $\omega(\lbrace c\rbrace)$ such that $\omega$ is a solution of our inverse problem. 
Therefore, first note that for every edge $\edge\in\edges$, we may choose $\omega_\edge$ such that  
\begin{align*}
 \phi_\edge(z,0) = P_\edge(z), \quad z\in\C, 
\end{align*}
 and such that the norms of the eigenfunctions are given by 
\begin{align}\label{eqnNormPhiDef}
 \|\phi_\edge(\mu,\cdot\,)\|_{H_0^1(I_\edge)}^{-2} = \frac{\eta_\mu}{\mu} \dot{P}_\edge(\mu)^{-2} \begin{cases} 1, & \mu\in\sigma_\edge\backslash\sigma, \\
                                                        \sum_{\fedge\in\edges_\mu} \Pi_{\mu,\fedge\edge}, &  \mu\in\sigma_\edge\cap\sigma.
                                                    \end{cases}
\end{align}                                    
In fact, the existence of a Borel measure $\omega_\edge$ on $I_\edge$ with all these properties is guaranteed by \cite[Theorem~4.1]{ThreeSpectra}. 
Moreover, the necessary fact that $\omega_\edge$ is finite near zero follows from \cite[Corollary~4.2]{ThreeSpectra}, since the sum 
\begin{align*}
 \sum_{\mu\in\sigma_\edge} \frac{\|\phi_\edge(\mu,\cdot\,)\|_{H_0^1(I_\edge)}^{2}}{\mu^3 \dot{P}_\edge(\mu)^{2}} \leq \sum_{\mu\in\sigma_\edge} \frac{1}{\mu^2\, \eta_\mu} 
\end{align*}
is finite due to the fact that~\eqref{eqnProdHN} is a Herglotz--Nevanlinna function. 
Furthermore, we may adjust the mass of $\omega$ in the central vertex, where we set
\begin{align*}
 \omega(\lbrace c \rbrace) = - \lim_{y\rightarrow\infty} \frac{1}{\I y} \frac{V(\I y)}{L} \prod_{\edge\in\edges} P_\edge(\I y)^{-1} \geq 0,
\end{align*} 
which is known to be non-negative since~\eqref{eqnProdHN} is a Herglotz--Nevanlinna function.  
With these choices, we have guaranteed that 
\begin{align}\label{eqnGinv}
 -G(z)^{-1} = -  \frac{V(z)}{L} \prod_{\edge\in\edges} P_\edge(z)^{-1}, \quad z\in\C\backslash\R. 
\end{align}
In fact, both of these functions have the same poles and residues by our definitions, bearing in mind~\eqref{eqnStar} and~\eqref{eqnRid}. 
Moreover, they also have the same growth along the imaginary axis by~\eqref{eqnGinvExp} and \cite[Corollary~10.8]{MeasureSL}. 
But since both of them are Herglotz--Nevanlinna functions (and coincide at zero), this already ensures~\eqref{eqnGinv}. 

 In conclusion, we have defined a weight measure $\omega\in\M(\G)$ such that the spectra associated with the individual edges are precisely the desired ones, that is, we have $\sigma(S_\edge) = \sigma_\edge$ for every edge $\edge\in\edges$. 
 As a consequence of~\eqref{eqnGinv}, this already guarantees that the spectrum associated with $\omega$ on the whole graph is $\sigma$, that is, $\sigma(S) = \sigma$. 
 In particular, the quantities $\kappa_\lambda$ and sets $\edges_\lambda$ for $\lambda\in\sigma$ coincide precisely with the ones defined in~\eqref{eqnkappalambda} and~\eqref{eqnedgeslambda}. 
 Finally, if $\lambda\in\sigma(S)$ is an eigenvalue such that the set $\edges_\lambda$ is not empty, then a calculation using~\eqref{eqnNormPhiDef} yields 
 \begin{align*}
  \frac{\|\phi_\edge(\lambda,\cdot\,)\|_{H_0^1(I_\edge)}^2}{\|\phi_\fedge(\lambda,\cdot\,)\|_{H_0^1(I_\fedge)}^2} \frac{\dot{\phi}_\fedge(\lambda,0)^2}{\dot{\phi}_\edge(\lambda,0)^2} = \frac{\sum_{\bedge\in\edges_{\lambda}} \Pi_{\lambda,\bedge\fedge}}{\sum_{\bedge\in\edges_\lambda} \Pi_{\lambda,\bedge\edge}} =   \frac{\sum_{\bedge\in\edges_{\lambda}} \Pi_{\lambda,\bedge\edge} \Pi_{\lambda,\edge\fedge}}{ \sum_{\bedge\in\edges_\lambda} \Pi_{\lambda,\bedge\edge}} = \Pi_{\lambda,\edge\fedge} 
 \end{align*}
 for all edges $\edge$, $\fedge\in\edges_\lambda$, which proves that $\Gamma_\lambda = \Pi_\lambda$. 

{\em Uniqueness.} 
Now suppose that $\tilde{\omega}\in\M(\G)$ is another solution of the inverse spectral problem and denote all corresponding quantities as the ones for $\omega$ but with an additional twiddle. 
Then from~\eqref{eqnW} we have 
\begin{align}\label{eqnMUniq} 
 \omega(\lbrace c\rbrace) z + \sum_{\edge\in\edges} \frac{\phi_\edge'(z,0)}{\phi_\edge(z,0)} = \tilde{\omega}(\lbrace c\rbrace) z + \sum_{\edge\in\edges} \frac{\tilde{\phi}_\edge'(z,0)}{\tilde{\phi}_\edge(z,0)}, \quad z\in\C\backslash\R,
\end{align}
which already shows $\tilde{\omega}(\lbrace c \rbrace) = \omega(\lbrace c\rbrace)$ in view of \cite[Corollary~10.8]{MeasureSL}.
 In order to show that we even have   
 \begin{align}\label{eqnMedgeuniq}
   \frac{\phi_\edge'(z,0)}{\phi_\edge(z,0)} = \frac{\tilde{\phi}_\edge'(z,0)}{\tilde{\phi}_\edge(z,0)}, \quad z\in\C\backslash\R, 
 \end{align}
 for every edge $\edge\in\edges$, it suffices to verify that their residues are the same (since they coincide at zero). 
 However, if $\mu\in\sigma(S_\edge)$ is no eigenvalue of $S$, then this is immediate from~\eqref{eqnMUniq} since then only one of the summands has a pole at $\mu$. 
 Otherwise, that is, when $\mu$ is an eigenvalue of $S$, then we may readily deduce this fact from~\eqref{eqnGresCM}. 
 Thus we conclude that~\eqref{eqnMedgeuniq} holds for every edge $\edge\in\edges$, which guarantees that $\tilde{\omega}_\edge =\omega_\edge$ on $I_\edge$ in view of, for example, \cite[Section~6.6]{dymc76}, \cite[Theorem~4.1]{ThreeSpectra}.
\end{proof}

As noted in the remark after Proposition~\ref{propGHN}, the condition on the function in~\eqref{eqnProdHN} being a Herglotz--Nevanlinna function only involves a particular kind of interlacing property for the given spectra; cf.\ \cite[Theorem~2.1]{gesi99}, \cite[Theorem~27.2.1]{le96}.  

Let us point out explicitly that Theorem~\ref{thmIP} comprises the fact that the weight measure $\omega\in\M(\G)$ is uniquely determined by the spectrum associated with the whole graph and all the spectra associated with the edges if and only if they are pairwise disjoint; cf.\ \cite[Theorem~5.1]{ThreeSpectra}. 
In fact, in this case all of the sets $\edges_\lambda$ are empty and hence no coupling matrices are present. 
However, certain parts of the weight measure may still be uniquely determined. 
For example, the weight $\omega(\lbrace c\rbrace)$ in the central vertex is always uniquely determined by the spectra in view of~\eqref{eqnGinvExp}. 
Also the measure $\omega_\edge$ on some edge $\edge\in\edges$ is already uniquely determined under the weaker condition that $\sigma(S)$ and $\sigma(S_\edge)$ are disjoint. 

It is also possible to relate the growth of the weight measure $\omega\in\M(\G)$ near some outer vertex to the growth of the associated spectral quantities. 
More precisely, we are able to give a precise criteria for the weight measure to be finite near some outer vertex in terms of our spectral data. 
To this end, let $\omega\in\M(\G)$ be some weight measure, denote all corresponding quantities as usual and set 
\begin{align}
 -\frac{1}{\eta_\mu} = - \lim_{z\rightarrow \mu} \frac{z-\mu}{G(z)}, \quad \mu\in\bigcup_{\edge\in\edges} \sigma(S_\edge). 
\end{align} 

\begin{corollary}\label{corRegCrit}
 For every edge $\edge\in\edges$, the  weight measure $\omega_\edge$ is finite near $l_\edge$ if and only if the sum  
 \begin{align}\label{eqnDecRho}
   \sum_{\mu\in\sigma(S_\edge)} & \frac{\eta_\mu}{\mu^2\, \dot{\phi}_\edge(\mu,0)^2} \begin{cases} 1, & \mu\not\in\sigma(S), \\ \sum_{\fedge\in\edges_\mu} \Gamma_{\mu,\fedge\edge}, & \mu\in\sigma(S), \end{cases}
\end{align}
 is finite. 
\end{corollary}

\begin{proof}
 This is immediate from \cite[Corollary~4.2]{ThreeSpectra} and~\eqref{eqnNormPhiDef}. 
\end{proof}

Hereby note that the residues of the function $G^{-1}$ are uniquely determined by the spectra.
Therefore, the criterium in Corollary~\ref{corRegCrit} is indeed formulated in terms of the given spectral quantities for our inverse spectral problem.

\section{Approximation with Stieltjes strings}\label{s6}

As a final result, we are going to show that the solution of our inverse problem can be approximated by Stieltjes strings \cite{pirotr13}, which are obtained by cutting off the given spectral data. 
To this end, we will first equip the space of weight measures $\M(\G)$ with the initial topology with respect to the linear functionals 
\begin{align}
 \omega \mapsto \int_\G f(x) \Trf(x)d\omega(x), \quad f\in C_0(\G),
\end{align}
on $\M(\G)$. 
Note that this is the weak$^\ast$ topology upon identifying $\M(\G)$ with a (weak$^\ast$ closed) subset of the dual of $C_0(\G)$. 
In fact, each measure $\omega\in\M(\G)$ can be regarded as the bounded linear functional
\begin{align}
 f \mapsto \int_\G f(x) \Trf(x) d\omega(x)
\end{align}
on $C_0(\G)$. 
As a consequence, the topology is metrizable on bounded (with respect to the operator norm of the corresponding functionals) subsets of $\M(\G)$ and convergent sequences in $\M(\G)$ are bounded. 
 In order to state the following results, consider some sequence of measures $\omega_n\in\M(\G)$, $n\in\N$ on the star graph $\G$. 
 We write $\omega_n\rightharpoonup^\ast\omega$ if this sequence converges to $\omega$ with respect to the weak$^\ast$ topology. 
 All quantities corresponding to the measures $\omega$ and $\omega_n$, $n\in\N$ are denoted as in the preceding sections but with an additional subscript $n\in\N$ for those of $\omega_n$.  

\begin{lemma}\label{lemConvGreen}
  If $\omega_n\rightharpoonup^\ast\omega$, then $G_n(z)\rightarrow G(z)$ locally uniformly for all $z\in\C\backslash\R$. 
\end{lemma}

\begin{proof}
First of all, we will show that the operators $S_n^{-1}$ converge to $S^{-1}$ in the weak operator topology. 
Indeed, Proposition~\ref{propWeakForm} shows that for $g\in H_0^1(\G)$ one has 
\begin{align*}
 \spr{S_n^{-1} g}{h}_{H_0^1(\G)} & = \int_\G g(x) h(x)^\ast d\omega_n(x) \rightarrow \int_\G g(x) h(x)^\ast d\omega(x) = \spr{S^{-1}g}{h}_{H_0^1(\G)}, 
\end{align*}
 at least for those functions $h\in H_0^1(\G)$ with compact support. 
However, functions with compact support are dense and our operators are uniformly bounded in view of Proposition~\ref{propTrace}. 
 Thus, we infer that this convergence even holds for all $h\in H_0^1(\G)$.  

In order to prove that the operators $S_n^{-1}$ converge even in the strong operator topology, fix some $g\in H_0^1(\G)$. 
Now one observes that the functions $S_n^{-1}g$ converge pointwise to $S^{-1}g$ as well as that $S_n^{-1} g$ is uniformly bounded in $H_0^1(\G)$. 
Thus we infer from the Arzel\`{a}--Ascoli theorem that the functions $S_n^{-1} g$ converge locally uniformly to $S^{-1} g$.  
If $g$ has compact support, then (use Proposition~\ref{propWeakForm} once more)
\begin{align*}
 \left| \|S_n^{-1} g\|_{H_0^1(\G)}^2 - \|S^{-1} g\|_{H_0^1(\G)}^2 \right|  = \left| \int_\G g\, S_n^{-1}g^\ast\,  d\omega_n - \int_\G g\, S^{-1}g^\ast\, d\omega \right| 
\end{align*}
can be estimated  by 
\begin{align*}
   \omega_n(\supp(g)) \sup_{x\in\supp(g)} |g(x)| & \left| S_n^{-1}g(x) - S^{-1}g(x)\right| \\ & \quad + \left| \int_\G g\, S^{-1}g^\ast\, d\omega_n - \int_\G g\, S^{-1}g^\ast\, d\omega \right|,
\end{align*}
which shows that $S_n^{-1}g$ converges to $S^{-1}g$ in $H_0^1(\G)$. 
Again, since our operators are uniformly bounded, we infer that $S_n^{-1}\rightarrow S^{-1}$ in the strong operator topology. 

 As a consequence, the operators $(S_n-z)^{-1}$ converge to $(S-z)^{-1}$ in the strong operator topology (see, e.g., \cite[Lemma~6.36 and Theorem~6.31]{tschroe}) for every $z\in\C\backslash\R$.  
 In order to finish the proof, just observe that 
 \begin{align*}
  \spr{(S-z)^{-1}\repc}{\repc}_{H_0^1(\G)} = \frac{G(z)-L}{z}, \quad z\in\C\backslash\R, 
 \end{align*}
 and that the functions $G_n$, $n\in\N$ are locally uniformly bounded on $\C\backslash\R$. 
\end{proof}

 With the aid of this auxiliary result, we are now going to show that $\omega$ can be approximated by particular Stieltjes strings. 
 These approximating discrete weight measures are obtained by cutting off the spectral data corresponding to $\omega$.
 To be precise, we suppose that the weight measures $\omega_n$ for every $n\in\N$ are given in terms of their spectral data by 
 \begin{align}\label{eqnSpecStiel}
  \sigma(S_n) & = \sigma(S) \cap (0,n), & \sigma(S_{n,\edge}) & = \sigma(S_\edge) \cap (0,n), \quad \edge\in\edges,
 \end{align}
 and by $\Gamma_{n,\lambda} = \Gamma_\lambda$ for every $\lambda\in\sigma(S_n)$ for which the set $\edges_{n,\lambda} = \edges_\lambda$ is not empty. 
 It is readily verified that these spectral data satisfy the conditions of Theorem~\ref{thmIP} and thus, the measures $\omega_n$ are well-defined. 
 In fact, they are known to be supported on finite sets since their corresponding spectra are finite. 
 Moreover, the positions and weights of the individual point masses can be written down explicitly in terms of the spectral data. 
 All this has been done quite recently in \cite{bopi08a}, \cite[Section~2]{pirotr13}. 
 
 Now the solution of the general inverse problem can be obtained as a limit (in the weak$^\ast$ topology) of these particular Stieltjes strings, as the following result shows.  

\begin{theorem}\label{thmAppro}
 If the weight measures $\omega_n$, $n\in\N$ are given in terms of their spectral data by~\eqref{eqnSpecStiel} and by $\Gamma_{n,\lambda} = \Gamma_\lambda$ for every $\lambda\in\sigma(S_n)$ for which the set $\edges_{n,\lambda}$ is not empty, then we have $\omega_n\rightharpoonup^\ast\omega$ and moreover, 
 \begin{align}\label{eqnTrfConv}
   \int_\G \Trf(x) d\omega_n(x) \rightarrow \int_\G \Trf(x) d\omega(x).
 \end{align} 
\end{theorem}

\begin{proof}
 Since the set $\lbrace \omega_n \,|\, n\in\N \rbrace$ is relatively compact in the weak$^\ast$ topology due to Proposition~\ref{propTrace}, it suffices to show that every convergent subsequence of $\omega_n$ converges to $\omega$.  
 Therefore, we pick some convergent subsequence $\omega_{n_k}$ of $\omega_n$ with limit $\tilde{\omega}\in\M(\G)$ and denote all quantities corresponding to $\tilde{\omega}$ with an additional twiddle.  
 In particular, for every given edge $\edge\in\edges$ we have  
 \begin{align*}
  \int_{I_\edge} f_\edge(x) x \biggl(1-\frac{x}{l_\edge} \biggr) d\omega_{n_k,\edge}(x) \rightarrow \int_{I_\edge} f_\edge(x) x \biggl( 1-\frac{x}{l_\edge} \biggr) d\tilde{\omega}_\edge(x)
 \end{align*}
 for each function $f_\edge\in C_0(I_\edge)$, as well as the limit 
 \begin{align*}
   \lim_{k\rightarrow\infty}\|\phi_{n_k,\edge}(\mu,\cdot\,)\|_{H_0^1(I_\edge)}^2 = \|\phi_\edge(\mu,\cdot\,)\|_{H_0^1(I_\edge)}^2, \quad \mu\in\sigma(S_\edge),
 \end{align*}
 in view of~\eqref{eqnNormPhiDef}. 
 Thus we infer from \cite[Proposition~3.3]{ThreeSpectra} that $\sigma(\tilde{S}_\edge) = \sigma(S_\edge)$ and  
 \begin{align*}
  \|\tilde{\phi}_\edge(\mu,\cdot\,)\|_{H_0^1(I_\edge)}^2 = \|\phi_\edge(\mu,\cdot\,)\|_{H_0^1(I_\edge)}^2, \quad \mu\in\sigma(S_\edge). 
 \end{align*}
 But this already shows that $\tilde{\omega}_\edge = \omega_\edge$ for every edge $\edge\in\edges$ in view of \cite[Theorem~4.1]{ThreeSpectra}. 
 Moreover, by Lemma~\ref{lemConvGreen} and the definition of our measures $\omega_n$ we have 
 \begin{align*}
  \tilde{G}(z) = G(z), \quad z\in\C\backslash\R,
 \end{align*} 
 which gives $\tilde{\omega}(\lbrace c\rbrace) = \omega(\lbrace c\rbrace)$.   
 As mentioned at the beginning of the proof, this implies $\omega_n\rightharpoonup^\ast\omega$ and the remaining claim~\eqref{eqnTrfConv} is immediate from~\eqref{eqnTraceForm}. 
\end{proof}

As a final remark, let us mention that under the assumptions of Theorem~\ref{thmAppro}, the corresponding operators $S_n^{-1}$ converge to $S^{-1}$ in the trace ideal norm. 
Indeed, the convergence in the strong operator topology is guaranteed by the proof of Lemma~\ref{lemConvGreen} and the (necessary) convergence of the trace norms is provided in~\eqref{eqnTrfConv}.

\bigskip
\noindent
{\bf Acknowledgments.}
 I gratefully acknowledge the kind hospitality of the {\em Institut Mittag-Leffler} (Djursholm, Sweden) during the scientific program on {\em Inverse Problems and Applications} in spring 2013, where this article was written.

\end{document}